\documentclass[12pt,reqno]{amsart}

\usepackage{amssymb, amsmath, amsthm, esint}
\usepackage{mathrsfs}
\usepackage{bbm,dsfont}
\usepackage{hyperref}
\usepackage{mathtools}
\usepackage{fullpage}
\usepackage{color}

\title[Discrete Stein--Wainger II]
{Discrete analogues of maximally modulated singular integrals of Stein-Wainger type: $\ell^p$ bounds for $p>1$}

\author[B. Krause]{Ben Krause}
\address{BK: Department of Mathematics, King's College London, WC2R 2LS, UK}
\email{ben.krause@kcl.ac.uk}

\author[J. Roos]{Joris Roos}
\address{JR: Department of Mathematical Sciences, University of Massachusetts Lowell, USA\\
\& School of Mathematics, The University of Edinburgh, Scotland, UK}
\email{joris\_roos@uml.edu}

\date{September 23, 2022}


\def\R{\mathbb{R}}
\def\N{\mathbb{N}}
\def\C{\mathbb{C}}
\def\Z{\mathbb{Z}}

\def\Q{\mathbb{Q}}

\theoremstyle{plain}
\newtheorem{mthm}{Theorem}
\newtheorem{mcor}[mthm]{Corollary}
\newtheorem{thm}{Theorem}[section]
\newtheorem{prop}[thm]{Proposition}
\newtheorem{lem}[thm]{Lemma}

\theoremstyle{remark}
\newtheorem*{remark}{Remark}

\numberwithin{equation}{section}
\newcommand{\D}{\mathrm{D}}

\begin{document}

\begin{abstract}
It is proved that certain discrete analogues of maximally modulated singular integrals
of Stein-Wainger type 
are bounded on $\ell^p(\Z^n)$
for all $p\in (1,\infty)$.
This extends earlier work of the authors concerning the case $p=2$. Some open problems for further investigation are briefly discussed.
\end{abstract}

\maketitle

\section{Introduction}
Let $n, d$ be positive integers and $K$ a Calder\'on--Zygmund kernel in $\R^n$. Consider the maximal operator
\begin{equation}\label{eqn:C-def}
\mathscr{C} f(x) = \sup_{\lambda\in\R^n}\Big|\sum_{y\in\Z^n\setminus\{0\}} f(x-y) e(\lambda |y|^{2d}) K(y)\Big|,\quad (x\in\Z^n),
\end{equation}
where $e(x)=e^{2\pi i x}$. This is a discrete analogue of a well-known maximal operator considered by Stein and Wainger \cite{SW01}. Our main result concerns $\ell^p$ bounds for this operator.

\begin{mthm}\label{thm:main} Let $p\in (1,\infty)$. Then there is a constant $C\in (0,\infty)$ such that
\begin{equation}\label{eqn:main}
\|\mathscr{C}f\|_{\ell^p(\Z^n)} \le C \|f\|_{\ell^p(\Z^n)}.
\end{equation}
The constant $C$ only depends on $p,d,n$ and $K$.
\end{mthm}

The case $p=2$ was proved in our previous paper \cite{KR}, building on earlier work by Krause and Lacey \cite{KL17}.
Here we shall heavily rely on arguments introduced in \cite{KR} and we recommend that the two papers be read side by side,
though all necessary preliminaries will be repeated here.

\subsection{Historical remarks and related problems}
The study of discrete analogues in harmonic analysis dates back to the work of Bourgain, who developed a theory of discrete polynomial maximal functions in the course of his breakthrough work on pointwise ergodic theorems, \cite{B0}, \cite{B2}, \cite{Bou89}. 
Subsequently Stein and Wainger \cite{SW9}, \cite{SW99} 
became interested in discrete analogues of singular integrals for their own sake.

Bourgain proved the following estimate.
Suppose that $P$ is a polynomial with integer coefficients, and $p\in (1,\infty]$. Then there exists a constant $C\in (0,\infty)$ so that
\begin{align*}
\| \sup_{N\ge 1} \frac{1}{N} \sum_{1\le y \leq N} |f(x-P(y))| \|_{\ell^p(\mathbb{Z})} \leq C \| f\|_{\ell^p(\mathbb{Z})}.
\end{align*}
Stein and Wainger \cite{SW9}, \cite{SW99} started investigating singular integral analogues such as 
\begin{align*}
    f \longmapsto \sum_{y\in\Z\setminus\{0\}} \frac{f(x-P(y))}{y},
\end{align*}
and higher-dimensional versions such as
\begin{align}\label{eqn:parabola}
    f\longmapsto \sum_{y\in\Z\setminus\{0\}} \frac{f(x_1-y,x_2-y^2)}{y},
\end{align}
the latter being a discrete analogue of a prototypical singular Radon transform.
Singular Radon transforms have been studied extensively in real harmonic analysis (see \cite{CNSW} and references contained therein).
The study of these discrete analogues was not motivated by ergodic-theoretic considerations, but by intrinsic interest in these operators.
Stein and Wainger \cite{SW99} proved $\ell^p$ estimates for $p\in (3/2,3)$ for a large class of polynomial Radon transforms such as these,
while estimates in the full range $p\in (1,\infty)$ where later established by Ionescu and Wainger \cite{IW06}.
This theory was extended and significantly refined by Mirek, Stein and Trojan \cite{MST15a}, \cite{MST15b} (also see references therein).
In view of these developments and Stein and Wainger's work on oscillatory integrals related to Carleson's theorem \cite{SW01}, it was then natural to ask for $\ell^p$
bounds for the maximal operator
\[ f \longmapsto \sup_{\lambda\in\R}\Big|\sum_{y\in\Z\setminus\{0\}} f(x-y) \frac{e(\lambda y^{2})}{y}\Big|, \]
which is our operator $\mathscr{C}$ if $n=d=1$ (the question was posed by Lillian Pierce at an AIM workshop in 2015).
We now describe some related problems that are still open.

A first consequence of our theorem concerns a variable-coefficient variant of \eqref{eqn:parabola}, given by
\begin{align}
    \mathcal{H}_v f(x)=\sum_{y\in\Z\setminus\{0\}} \frac{f(x_1-y,x_2-v(x)y^2)}{y},
\end{align}
where $v:\Z^2\to\Z$ is an arbitrary function. 
This is a discrete analogue of a real-variable operator 
studied by Guo, Hickman, Lie and one of the authors \cite{GHLR}.

By taking a partial Fourier transform in the $x_2$ variable and applying Theorem \ref{thm:main} with $p=2$,
the following result is obtained.
\begin{mcor}
There exists a constant $C\in (0,\infty)$ so that
for all $v:\Z^2\to \Z$ satisfying $v(x_1,x_2)=v(x_1,0)$ for all $(x_1,x_2)\in\Z^2$,
\begin{align*}
\| \mathcal{H}_v f \|_{\ell^2(\mathbb{Z}^2)} \leq C \|f\|_{\ell^2(\mathbb{Z}^2)}.
\end{align*}
\end{mcor}
It would be interesting to prove such estimates for $p\not=2$. Discrete analogues of maximal functions associated with variable curves, such as
\begin{align}
    f\longmapsto \sup_{N\ge 1} \frac1{N}\sum_{1\le y\le N} |f(x_1-y,x_2-v(x)y^2)|.
\end{align}
are also of interest.

In view of Lie's quadratic Carleson theorem \cite{Lie}, a long term goal is to study 
modulation invariant discrete analogues such as
\begin{equation}\label{eqn:superhard}
f\longmapsto \sup_{\lambda,\mu\in\R} \Big| \sum_{y\in\Z\setminus\{0\}} f(x-y) \frac{e(\lambda y+\mu y^2)}{y} \Big|.
\end{equation}
If the supremum is restricted to one of the variables $\lambda$ or $\mu$, then $\ell^p$ bounds are known, though they are obtained
by entirely different methods: if $\lambda$ is dropped, we are in the situation of the present paper; if $\mu$ is dropped, then 
familiar linear Carleson theory can be applied (see e.g. Lacey-Thiele \cite{LT00}).
It is not clear how the two methods could be combined.

An interesting model problem is to determine whether the operator
\[ f\longmapsto \sup_{\lambda\in\R} \Big| \sum_{y\in\Z\setminus\{0\}} f(x-y^3) \frac{e(\lambda y^3)}{y} \Big| \]
is bounded on $\ell^2(\Z)$, as this is a prototypical example of an operator that combines modulation invariance and Radon behavior. 
A further obstacle in \eqref{eqn:superhard} is the multi-parameter setting. This is an appealing challenge even in non-modulation invariant cases.
We hope to return to these questions in the future.

\subsection*{Acknowledgements}
J.R. would like to thank the Hausdorff Research Institute for Mathematics in Bonn
and the organizers of the trimester program ``Harmonic Analysis and Analytic Number Theory''
for a pleasant work environment.

\section{Preliminaries}\label{sec:prelim}

\subsection{Notation}

Given a bounded function $m$ on $\R^n$ we denote the Fourier multiplier on $\R^n$
associated with $m$ by $m(D)$. If $m$ is $1$-periodic, $m(D)$ may also stand for the Fourier multiplier on $\Z^n$. It will
always be clear from context which one is meant.

Given integers (or vectors of integers) $a_1,\dots,a_k$, we denote their greatest common divisor by $(a_1,\dots,a_k)$
and for a positive integer $q$ we denote the set of non-negative integers smaller than $q$ by $[q]$.

For non-negative quantities $A$ and $B$, we use the notation $A\lesssim B$ to indicate existence of a constant $C$ so that $A\le C\cdot B$,
where $C$ may depend on various parameters, which may sometimes be indicated by subscripts. 
Similarly, $A\approx B$ means that $A\lesssim B$ and $B\lesssim A$, wheras $A\asymp B$ means that $2^{-1}B\le A\le 2B$.
All constants throughout are allowed to depend on $d$ and $n$.

\subsection{Exponential sum estimates}

Let $D$ be a positive integer. For a set of coefficients $\xi=(\xi_\alpha)_{1\le |\alpha|\le D}$ (with $\alpha\in\N^n_0$ a multiindex),
define the polynomial
\[ P(\xi; x) = \sum_{1\le |\alpha|\le D} \xi_\alpha x^\alpha. \]
Further let $R\ge 2$ and fix a smooth function $\varphi$ on $\R^d$ such that 
\[ |\varphi(x)|\le 1,\quad |\nabla \varphi(x)|\le (1+|x|)^{-1} \]
holds for all $x\in\R^n$. Finally, let $c_0>0$ and let $\omega$ be a convex set contained in the ball of radius $c_0 R$ centered at the origin.
We are interested in the exponential sum
\[ S_{R} = \sum_{x\in\Z^n\cap\omega} e(P(\xi; x)) \varphi(x). \]
The triangle inequality implies that $|S_{P,R,\varphi,\omega}| \lesssim_{c_0,n} R^n$.
Heuristically, we can expect an improvement to this trivial estimate if at least one of the coefficients $\xi_\alpha$ is in some sense badly approximated by rationals with small denominators.
We shall require two quantitatively distinct manifestations of this principle.
The first is due to Stein and Wainger \cite[Proposition 3]{SW99}.
\begin{prop}\label{prop:expsumpower}
For every $\varepsilon>0$ there exists $\delta=\delta(\varepsilon,n,D)>0$ such that if $\alpha_0$ with $1\le |\alpha_0|\le D$ is such that
\[ |\xi_{\alpha_0}-\tfrac{a}q| \le \tfrac1{q^2},\quad R^{\varepsilon}\le q\le R^{|\alpha_0|-\varepsilon} \]
for a reduced rational $a/q$, then
\[ |S_R| \le C R^{n-\delta}, \]
where the constant $C$ depends only on $n, D, c_0, \varepsilon$.
\end{prop}

The following refinement is due to Mirek, Stein and Trojan \cite[Theorem 3.1]{MST15a}.

\begin{prop}\label{prop:expsumlog}
For every $\gamma>0$ there exists $\delta=\delta(\gamma,n,D)>0$ such that if $\alpha_0$ with $1\le |\alpha_0|\le D$ is such that
\[ |\xi_{\alpha_0}-\tfrac{a}q| \le \tfrac1{q^2},\quad (\log R)^{\delta}\le q\le R^{|\alpha_0|} (\log R)^{-\delta} \]
for a reduced rational $a/q$, then
\[ |S_R| \le C R^n (\log R)^{-\gamma}, \]
where the constant $C$ depends only on $n, D, c_0, \gamma$.
\end{prop}

\subsection{Ionescu-Wainger theory}

Call a subset $\Theta \subset \R^n$ {\em periodic} if $x+\Theta=\Theta$ for all $x\in\Z^n$.
Given a bounded function $m$ on $\R^n$ and a rational periodic frequency set $\Theta\subset\Q^n$, consider the periodic multi-frequency multiplier
\[  \Delta_\Theta[m](\xi) = \sum_{\theta\in\Theta} m(\xi-\theta).  \]
The idea is roughly that if $\Theta$ has appropriate arithmetic properties and $m$ has sufficiently small support, then the $L^p$ theories of $\Delta_\Theta[m]$ (as a multiplier on $\Z^n$) and $m$ (as a multiplier on $\R^n$) are closely related.
Define 
\[ \mathscr{B}_N = \{\tfrac{b}q\in\Q^n\,:\,(b,q)=1,\,1\le q\le N\}  \]
and let $\eta$ be a smooth function on $\R^n$ supported on $\{|\xi|\le 1\}$ and equal to one on $\{|\xi|\le 1/2\}$. Denote $\eta_\lambda(\xi)=\eta(\lambda^{-1}\xi)$.

\begin{prop}\label{prop:iw}
Let $p\in (1,\infty)$ and $r$ a positive integer with $(2r)'\le p\le 2r$.
Further let $m$ be a bounded function on $\R^n$.
Assume that $A>0$ is such that for all $f\in L^{2r}(\R^n)$,
\[ \| m(D) f\|_{L^{2r}(\R^n)} \le A \|f\|_{L^{2r}(\R^n)}. \]
Let $\rho>0$ and $N\ge 2^{\lfloor \frac{2}\rho\rfloor+1}$. Then there exists 
a periodic set $\mathscr{U}_N\subset\Q^n$ such that
\[ \mathscr{B}_N \subset \mathscr{U}_N \subset \mathscr{B}_{2^{N^\rho}}  \]
and for all $f\in \ell^p(\Z^n)$,
\[ \Big\| \Delta_{\mathscr{U}_N} [m_\nu \eta_{2^{-N^{2\rho}}}](D) f\Big\|_{\ell^p(\Z^n)} \lesssim_{\rho,p,r} A \|f\|_{\ell^{p}(\Z^n)} \]
(In particular, the implicit constant is independent of $N$.)
\end{prop}

\begin{remark}
This result was first proved by Ionescu and Wainger \cite{IW06}; it originally featured a logarithmic loss in $N$. This loss was later improved by Mirek \cite{Mir18} and recently Tao \cite{Tao21} showed that no loss in $N$ occurs. (However, this is not crucial for our application.)
\end{remark}

\subsection{A numerical inequality}

The following is a variant of an inequality that proved useful in several recent works such as \cite{MST15a}, \cite{MST15b}, \cite{MT16}.

\begin{lem}\label{lem:rmineq}
Let $r\in [1,\infty), s\in\N$ and $(a_j)_{0\le j\le 2^s}$ a family of complex numbers.
Then for all $0\le j, j_0\le 2^s$,
\[ |a_j| \le |a_{j_0}| + 2^{1/r'} \sum_{l=0}^s \Big( \sum_{\kappa=0}^{2^{s-l}-1} |a_{\kappa 2^l} - a_{(\kappa+1)2^l}|^r \Big)^{1/r}. \]
\end{lem}

This is an immediate consequence of the observation that for every $j_0\le j_1$, the interval $[j_0,j_1)$
can be partitioned into dyadic intervals in such a way that each dyadic length occurs at most twice.
For the proof of this observation and further details we refer to \cite[Lemma 1]{MT16}. 

\subsection{Multiplier approximations}

We decompose $K(x)=\sum_{j\ge 1} K_j(x)$ so that 
for all $j\ge 1$, $K_j$ is supported in $\{x\,:\,|x|\le 2^{j+1}\}$,
and satisfies the standard estimates
\[ |K_j(x)|\lesssim 2^{-jn},\quad |\nabla K_j(x)|\lesssim 2^{-j(n+1)}. \]
Moreover, let $K_j$ be supported in $\{x\,:\,|x|\ge 2^{j-1}\}$ for $j\ge 2$.
Define associated periodic multipliers by
\[ m_{j,\lambda}(\xi) = \sum_{y\in\Z^n} e(\lambda |y|^{2d} + \xi\cdot y) K_j(y). \]
We recall a basic approximation result for $m_{j,\lambda}(\xi)$ in the spirit of Bourgain \cite{Bou89}.
To  do this, define the exponential sums
\[ S(\tfrac{a}q, \tfrac{b}{q}) = q^{-n} \sum_{r\in [q]^n} e(\tfrac{a}q |r|^{2d}+\tfrac{b}q\cdot r), \]
where $\frac{a}q\in \Q, \frac{b}{q}\in \Q^n$ are rationals with $(a,b,q)=1$
and further the oscillatory integrals,
\[ \Phi_{j,\lambda}(\xi) = \int_{\R^n} e(\lambda |y|^{2d}+\xi\cdot y) K_j(y) dy. \]
The basic result, which was proved in \cite[\S 2]{KR} now reads as follows.

\begin{lem}\label{lem:multapprox}
Let $j$ and $q$ be positive integers so that $q\le 2^{j-2}$, $a\in\Z, b\in\Z^n$ and $(a,b,q)=1$.
Assume that $\lambda\in\R, \xi\in\R^n$ satisfy
\[ |\lambda-\tfrac{a}q|\le \delta 2^{-(2d-1)j},\quad |\xi-\tfrac{b}q|\le \delta \]
with $\delta\in (2^{-j},1)$. Then
\[ m_{j,\lambda}(\xi) = S(\tfrac{a}q,\tfrac{b}q)\Phi_{j,\lambda-\frac{a}q}(\xi-\tfrac{b}q)+O(q\delta) \]
and the implicit constant depends only on $d,n,K$.
\end{lem}

Another key piece of information is that for $(a,b,q)=1$,
\begin{equation}\label{eqn:bcyclic}
(a,q)>1\quad\mathrm{implies}\quad S(\tfrac{a}q,\tfrac{b}q)=0.
\end{equation}
For the proof we refer to \cite[Lemma 2.3]{KR}.

\section{Reduction to major arcs}
This section closely follows the corresponding reductions in our previous paper \cite{KR}, with some adaptations required to facilitate the application of Ionescu-Wainger theory.

We begin with observing the estimate
\begin{equation}\label{eqn:trivunif}
\|\sup_{\lambda\in\R} |m_{j,\lambda}(D)f|\|_{\ell^p(\Z^n)} \lesssim \|f\|_{\ell^p(\Z^n)},
\end{equation} 
valid for all $p\in[1,\infty]$ and $j\ge 0$. This follows from the triangle inequality and Young's convolution inequality. 

The next step is a non-trivial improvement for the case of $\lambda$ away from rationals with small denominators. For a real number $N>1$, define
\[ \mathfrak{A}(N) = \{ \tfrac{a}q\in\Q\,:\,(a,q)=1,\,1\le q\le N\}, \]
\[ X_{j,M} = \bigcup_{\alpha\in \mathfrak{A}_{j^M}} \{\lambda\in\R\,:\,|\lambda-\alpha|\le 2^{-2dj} j^M \}.  \] 

\begin{prop}\label{prop:tts1}For every $p\in (1,\infty)$ and every $\kappa>0$ there exists $M=M(p,\kappa)>0$ large enough so that for all $j\ge 1$,
\[ \| \sup_{\lambda\not\in X_{j,M}} |m_{j,\lambda}(\D)f| \|_{\ell^p(\Z^n)} \lesssim j^{-\kappa} \|f\|_{\ell^p(\Z^n)}. \]
\end{prop}

By interpolating with \eqref{eqn:trivunif}, it suffices to show the claim for $p=2$. The proof of this is in essence identical with that of \cite[Prop. 3.1]{KR}; we provide details for the necessary changes in \S \ref{sec:TTstar}.
Before we describe the major arc approximation for the remaining part of our operator, let us introduce some convenient notations. For $s\ge 1$ define
\[ \mathcal{A}_s = \{ \tfrac{a}q\in\Q\,:\,(a,q)=1, q\in [2^{s-1}, 2^s) \}. \]
For a reduced rational $\alpha=\frac{a}q\in\mathcal{A}_s$, $\rho>0$, a bounded function $m$ on $\R^n$ and $\xi\in\R^n$, write
\begin{equation}\label{eqn:Lsmdef} \mathscr{L}_{s,\alpha,M}[m](\xi) = 
\sum_{\beta\in \frac1q\Z^n} S(\alpha,\beta)m(\xi-\beta)\chi_{s,M }(\xi-\beta),
\end{equation}
where $\chi_{s,M}(\xi)=\chi(2^{4s 2^{s/(2M)}} \xi)$ with $\chi$ a smooth radial function satisfying $0\le \chi\le 1$ that is equal to one on $\{|\xi|\le 1/4\}$ and supported on $\{|\xi|\le 1/2\}$.
The definition \eqref{eqn:Lsmdef} is motivated by Lemma \ref{lem:multapprox}.
A major obstacle is that the range of frequencies $\beta$ depends on the denominator $q$ of $\alpha$, which is essentially the modulation parameter over which a supremum is taken.
A redeeming quality is that we can let $\beta$ run over a product of cyclic subgroups.
This exploits the relation \eqref{eqn:bcyclic} and is crucial at various points in the argument.
Define
\[ L^s_{j,\lambda,M} = \mathscr{L}_{s,\alpha,M}[\Phi_{j,\lambda-\alpha,M}], \]
where $\alpha$ is the unique $\alpha\in\mathcal{A}_s$ so that $|\lambda-\alpha|\le 2^{-4s 2^{s/(2M)}}$, or an arbitrary element of the complement of $\mathcal{A}_s$ if no such $\alpha$ exists and
\[\Phi_{j,\nu,M} = \Phi_{j,\nu}\cdot \mathbf{1}_{|\nu|\le 2^{-2dj}j^M}. \]
Finally, decompose
\begin{equation}\label{eqn:LEjdef}
m_{j,\lambda}\mathbf{1}_{X_{j,M}} = \sum_{s\ge 1\,:\,2^s\le j^M} L^s_{j,\lambda,M} + E_{j,\lambda,M},
\end{equation}
which we take to be the definition of the error term, $E_{j,\lambda,M}$.

\begin{prop}\label{prop:error}
	For every $p\in (1,\infty)$ there exists $\gamma_p>0$ so that
	for every $M>0$ and every $j\ge 1$,
	\[ \|\sup_{\lambda \in X_{j,M}} |E_{j,\lambda,M}(D)|\|_{\ell^p(\Z^n)} \lesssim 2^{-\gamma_p j} \|f\|_{\ell^2(\Z^n)}. \]
\end{prop}

This is proved by a standard Sobolev embedding-type argument, relying on Lemma \ref{lem:multapprox}, Proposition \ref{prop:expsumpower} and standard estimates for oscillatory integrals.
The proof coincides with that of
\cite[Prop. 3.2]{KR}, but we provide some details in \S \ref{sec:errorest}.\\

Fix $p\in (1,\infty)$ and choose $M=M(p,2)$ as in Proposition \ref{prop:tts1}.
To establish Theorem \ref{thm:main} it now suffices to show that
\begin{equation}\label{eqn:mainclaim}
\Big\| \sup_{\lambda\in\R} \Big|\sum_{j\ge 0} \sum_{s\ge 1\,:\,2^s\le j^M} L^s_{j,\lambda,M}(D)f\Big| \Big\|_{\ell^p(\Z^n)} \lesssim \|f\|_{\ell^p(\Z^n)}.
\end{equation}

\section{Main argument}\label{sec:main}
To show \eqref{eqn:mainclaim} we will show existence of $\gamma_p>0$ so that for {\em every} $M>0$, and every $s\ge 1$,
\begin{equation}\label{eqn:penult}
\Big\| \sup_{\lambda\in\R} \Big|\sum_{j\ge 0\,:\,2^s\le j^M} L^s_{j,\lambda,M}f \Big| \Big\|_{\ell^p(\Z^n)} \lesssim 2^{-\gamma_p s} \|f\|_{\ell^p(\Z^n)},
\end{equation}
where the implicit constant may depend on $d,n$ and $M$.
Since the value of $M>0$ is not important in this section, we 
suppress it from notation and write
\[\mathscr{L}_{s,\alpha}=\mathscr{L}_{s,\alpha,M}.\]
See \eqref{eqn:Lsmdef} for the definition of $\mathscr{L}_{s,\alpha,M}$. Define the auxiliary multiplier,
\begin{equation}
\mathscr{L}_{s}^\sharp[m](\xi) = \sum_{\beta\in \mathscr{U}_{2^s}} m(\xi-\beta)\widetilde{\chi}_s(\xi-\beta) = \Delta_{\mathscr{U}_{2^s}}[m \widetilde{\chi}_s],
\end{equation}
where $\mathscr{U}_{2^s}$ is a set as provided by Proposition \ref{prop:iw} and $\widetilde{\chi}$ is a smooth radial function with $0\le \widetilde{\chi}\le 1$
that is equal to one on $\{|\xi|\le 1/2\}$ and supported on $\{|\xi|\le 1\}$ (so that $\chi\cdot\widetilde{\chi}=\chi$), and $\widetilde{\chi}_s(\xi)=\widetilde{\chi}(2^{4s 2^{s/(2M)}}\xi)$.
Observe the crucial factorization
\begin{equation}\label{eqn:factorization}
\mathscr{L}_{s,\alpha}[m] = \mathscr{L}_{s,\alpha}[1]\cdot \mathscr{L}^\sharp_{s}[m],
\end{equation}
which follows from disjoint supports of the shifted cutoff functions, and uses the fact that $\frac1{q}\Z^n\subset \mathscr{U}_{2^s}$ for all $q\le 2^s$.
We now begin with the main sequence of lemmas that will pave the way for the proof of \eqref{eqn:penult}.
The first should be regarded as a trivial estimate, but is crucial in proving the result for all $p>1$.

\begin{lem}\label{lem:triv}
(i) Let $m:\R^n\to \C$ be a bounded function and $\alpha\in\Q$. Then
for every $y\in\Z^n$,
\begin{equation}\label{eqn:kernel}
\mathcal{F}^{-1}_{\Z^n}[\mathscr{L}_{s,\alpha}[m]] (y) = e(\alpha |y|^{2d}) \mathcal{F}^{-1}_{\R^n}[m \chi_s](y).
\end{equation}
(ii) Let $\mathcal{I}$ be a set and $(m_u)_{u\in\mathcal{I}}$ a family of bounded measurable functions on $\R^n$. Then for all $p\in [1,\infty]$ and all $s\ge 1$,
\[ \|\sup_{u\in\mathcal{I}} \sup_{\alpha\in\Q} |\mathscr{L}_{s,\alpha}[m_u](D)f|\|_{\ell^p(\Z^n)} \lesssim  \|\sup_{u\in\mathcal{I}} |\mathcal{F}^{-1}_{\R^n} [\chi_s m_u]|\|_{\ell^1(\Z^n)} \|f\|_{\ell^p(\Z^n)}. \]
\end{lem}

\begin{proof}
(i)
Writing $\alpha=a/q$ with $(a,q)=1$ we compute the kernel
\[ \mathcal{F}_{\Z^n}^{-1} [\mathscr{L}_{s,\alpha}] (y) = \sum_{\beta\in \frac1q \Z^n\cap [0,1)^n} S(\alpha,\beta) \int_{\R^n} e(y\cdot \xi) m(\xi-\beta) \chi_s(\xi-\beta) d\xi \]
\[ = \Big[\sum_{\beta\in \frac1q [q]^n} S(\alpha,\beta) e(y\cdot \beta)\Big] \mathcal{F}_{\R^n}^{-1}[m \chi_s](y).\quad(y\in\Z^n) \]
Write $\beta=\frac{b}{q}$ with $b\in [q]^n$. Then,
\[
\sum_{\beta \in \frac1{q} [q]^n} S(\alpha,\beta) e(y\cdot \beta) = \frac1{q^n} \sum_{r\in[q]^n} e(\tfrac{a}q |r|^{2d}) \Big( \sum_{b\in [q]^n} e(\tfrac{b}{q}\cdot (y+r)) \Big)
\]
Changing variables $r\mapsto r-y$ and using periodicity, the previous equals
\[ = \frac1{q^n} \sum_{r\in [q]^n} e(\tfrac{a}{q}|r-y|^{2d}) q^n \mathbf{1}_{r=0} = e(\tfrac{a}{q} |y|^{2d}), \]
which shows \eqref{eqn:kernel}.

(ii) Part (i) implies the pointwise estimate
\[ \sup_{u\in\mathcal{I}} \sup_{\alpha\in\Q} |\mathscr{L}_{s,\alpha}[m_u](D)f(x)|  \le (|f|*(\sup_{u\in\mathcal{I}} |\mathcal{F}_{\R^n}^{-1}[\chi_s m_u]|)(x), \]
which in turn implies the claim. 
\end{proof}

\begin{lem}\label{prop:majorarc1}
Let $\mathcal{I}$ be a countable set.
For every $p\in (1,\infty)$ there exists $\gamma_p>0$ such that for all $s\ge 1$,
\[ \Big\| \Big(\sum_{\nu\in\mathcal{I}} \sup_{\alpha\in\mathcal{A}_s} |\mathscr{L}_{s,\alpha}[1](D) f_\nu|^2\Big)^{1/2}\Big\|_{\ell^p(\Z^n)} \lesssim 
2^{-\gamma_p s} \Big\| \Big(\sum_{\nu\in\mathcal{I}} |f_\nu|^2\Big)^{1/2}\Big\|_{\ell^p(\Z^n)} \]
\end{lem} 

\begin{proof}
For $p=2$ this was proved in \cite[Prop. 3.3]{KR} (the cutoff functions $\chi_s$ are localized less narrowly there, but this only makes the claim easier; the same proof works verbatim).
For $p\in [1,\infty]$ we may use Lemma \ref{lem:triv} (i) to dominate 
the left-hand side by
\[ \Big\| \Big( \sum_{\nu\in\mathcal{I}} (|f_\nu|*|\phi_s|)^2 \Big)^{1/2} \Big\|_{\ell^p(\Z^n)} \lesssim \Big\| \Big(\sum_{\nu\in\mathcal{I}} |f_\nu|^2\Big)^{1/2}\Big\|_{\ell^p(\Z^n)}, \] 
where the last estimate is standard; it follows from Khinchine's and Minkowski's inequalities and the fact that $\|\phi_s\|_{\ell^1}\approx 1$.
Interpolating with $p=2$ yields the claim.
\end{proof}

The previous lemma will allow us to prove the following
 $\ell^p$ analogues of \cite[Lemma 7.1]{KR} and
 \cite[Lemma 7.2]{KR}.

\begin{lem}\label{prop:Lemma7.1}
Let $\mathcal{I}$ be a countable set and $(m_\nu)_{\nu\in\mathcal{I}}$ a family of bounded functions.
Let $p\in (1,\infty)$ and $r$ a positive integer with $(2r)'\le p\le 2r$.
Assume that for all $(\epsilon_\nu)_{\nu\in\mathcal{I}}$ with $\epsilon_\nu\in \{\pm 1\}$,
\begin{equation}\label{eqn:prop35asm}
\Big\|
\sum_{\nu\in\mathcal{I}} \epsilon_\nu m_\nu(D) f
\Big\|_{L^{2r}(\R^n)} \le A \|f\|_{L^{2r}(\R^n)}.
\end{equation}
Then there exists $\gamma_p>0$ such that for all $s\ge 1$,
\[ \Big\| \Big(\sum_{\nu\in\mathcal{I}} \sup_{\alpha\in\mathcal{A}_s} |\mathscr{L}_{s,\alpha}[m_\nu](D) f|^2\Big)^{1/2}\Big\|_{\ell^p(\Z^n)} \lesssim 
2^{-\gamma_p s} A \|f\|_{\ell^p(\Z^n)}. \]	
\end{lem}

\begin{proof}
By the factorization \eqref{eqn:factorization} and Lemma \ref{prop:majorarc1}
we estimate the left-hand side by a constant times
\[ 2^{-\gamma_p s} \Big\| \Big( \sum_{\nu\in\mathcal{I}} |\mathscr{L}^\sharp_s[m_\nu](D)f|^2 \Big)^{1/2} \Big\|_{\ell^p(\Z^n)} \]
with $\gamma_p$ as in Lemma \ref{prop:majorarc1}.
By Khinchine's inequality,
the previous is
\[ \lesssim 2^{-\gamma_p s} \Big( \mathbb{E}_{\epsilon} \Big\|
\mathscr{L}^\sharp_s\Big[\sum_{\nu\in\mathcal{I}} \epsilon_\nu m_\nu\Big](D)f
\Big\|^p_{\ell^p(\Z^n)} \Big)^{1/p},\]
with $\mathbb{E}_\epsilon$ denoting expectation taken over i.i.d. random variables $\epsilon_\nu$ each with $\mathbb{P}(\epsilon_\nu=\pm 1)=\tfrac12$.
By the Ionescu-Wainger multiplier theorem (that is, Proposition \ref{prop:iw}), the previous display is dominated by
\[ \lesssim 2^{-\gamma_p s} \sup_{\epsilon} \Big\|
\sum_{\nu\in\mathcal{I}} \epsilon_\nu m_\nu(D)
\Big\|_{L^{2r}\to L^{2r}}  \|f\|_{\ell^p(\Z^n)},\]
where $r$ is a positive integer with $(2r)'\le p\le 2r$, and the supremum is over all choices of $\epsilon_\nu \in \{ \pm 1\}$. Applying the assumption \eqref{eqn:prop35asm} gives the claim.
\end{proof}

\begin{lem}\label{prop:Lemma7.2}
For $j\ge 1$ let $\mathcal{K}_j$ be a mean zero $C^1$ function supported on $\{|x|\asymp 2^j\}$
so that
\[ 2^{jn}|\mathcal{K}_j(x)| + 2^{j(n+1)}|\nabla \mathcal{K}_j(x)|\lesssim 1 \]
for all $x\in\R^n$ and $j\ge 1$. Let $\mathcal{K}^{a,b} = \sum_{a\le j<b}\mathcal{K}_j$.
Then for every $p\in (1,\infty)$ there exists $\gamma_p>0$ such that for all $s\ge 1$,
\[ \|\sup_{J\ge 1} \sup_{\alpha\in\mathcal{A}_s} |\mathscr{L}_{s,\alpha}[ \widehat{\mathcal{K}^{0,J}}](D)f| \|_{\ell^p(\Z^n)} \lesssim 2^{-\gamma_p s} \|f\|_{\ell^p(\Z^n)}. \]
\end{lem}
We postpone the proof to the end of this section.
It closely resembles that of \cite[Lemma 7.2]{KR}.
We are now ready to show \eqref{eqn:penult}, which finishes the proof of Theorem \ref{thm:main}.

\begin{proof}[Proof of \eqref{eqn:penult}]
We fix $s\ge 1$ and begin by splitting the operator into three components, similarly as in \cite{KR}.
Define 
\[ \mathcal{J}_{\ell,\mu} = \{ j\,:\,2^s\le j^M,\;|\mu| 2^{2dj}\asymp 2^\ell \},\quad \widetilde{\Phi}_{\ell,\mu} = \sum_{j\in \mathcal{J}_{\ell,\mu}} \Phi_{j,\mu}, \]
\[ \mathcal{L}_1 = \{ \ell\in\Z\,:\,\ell\ge s\},\quad \mathcal{L}_2 = \{ \ell\in\Z\,:\,-s<\ell< s\},  \]
\[ \mathcal{L}_3 = \{ \ell\in\Z\,:\,\ell\le -s \}. \]
Then it suffices to bound each of the terms
\[ \|\sup_{\alpha\in\mathcal{A}_s} \sup_{\mu} |\sum_{\ell\in\mathcal{L}_i} \mathscr{L}_{s,\alpha}[\widetilde{\Phi}_{\ell,\mu}](D) f| \|_{\ell^p(\Z^n)} \]
for $i=1,2,3$, separately.

The high frequency component, $i=1$, is handled using 
a theorem of Stein and Wainger \cite{SW01}, which implies that for every fixed $\ell\in\mathcal{L}_1$,
\[ \| \sup_{\mu\in\R} |\widetilde{\Phi}_{\ell,\mu}(D)g|\|_{L^p(\R^n)} \lesssim 2^{-\gamma_p \ell} \|g\|_{L^p(\R^n)} \]
for some $\gamma_p>0$. 
A simple transference argument and Lemma \ref{lem:triv} (i) then give the desired exponential decay in $s$.

The low frequency component, $i=3$, can be written as
\[\|\sup_{\alpha\in\mathcal{A}_s} \sup_\mu |\mathscr{L}_{s,\alpha}[\sum_{J_-\le j\le J_{+,\mu}}\Phi_{j,\mu}](D)f|\|_{\ell^p(\Z^n)},\]
where $J_-$ is the smallest integer $j$ so that $2^s\le j^M$ and $J_{+,\mu}$ is the largest integer $j$ so that $|\mu| 2^{2dj}<2^{-s+1}$.
The oscillatory integrals $\Phi_{j,\mu}$ do not experience significant oscillation in this case, so that it is reasonable to bound the previous by
\begin{equation}\label{eqn:lowfreqpart1} 
\|\sup_{\alpha\in\mathcal{A}_s} \sup_\mu |\mathscr{L}_{s,\alpha}[\sum_{J_-\le j\le J_{+,\mu}}(\Phi_{j,\mu}-\widehat{K_j})](D)f|\|_{\ell^p(\Z^n)},
\end{equation}
plus
\[\|\sup_{\alpha\in\mathcal{A}_s} \sup_{J\ge J_-} |\mathscr{L}_{s,\alpha}[\sum_{J_-\le j\le J} \widehat{K_j}](D)f|\|_{\ell^p(\Z^n)}.\]
Lemma \ref{prop:Lemma7.2} shows that the latter term is bounded by $2^{-\gamma_p s}\|f\|_{\ell^p(\Z^n)}$.
Also note that by definition of $J_{+,\mu}$,
\[ \Big\|\sum_{J_-\le j\le J_{+,\mu}} (\mathcal{F}^{-1}_{\R^n}[\Phi_{j,\mu}] - K_j)\Big\|_{L^1(\R^n)} \lesssim \sum_{J_-\le j\le J_{+,\mu}} |\mu| 2^{2dj} \lesssim 2^{-s}. \]
In view of Lemma \ref{lem:triv} (ii) this shows that \eqref{eqn:lowfreqpart1} is dominated by a constant times $2^{-s} \|f\|_{\ell^p(\Z^n)}$.

The heart of the problem lies in the analysis of the
 intermediate frequency component, $i=2$, which exhibits stationary phase behavior.
Since $\# \mathcal{L}_2\lesssim s$, it suffices to establish exponential decay for each $\ell\in\mathcal{L}_2$ separately.
For $p=2$, it is shown in \cite[\S 7.2]{KR} (the slightly different definition of the present multipliers does not affect the argument) that for every $\ell\in\mathcal{L}_2$,
\[ \|\sup_{\alpha\in\mathcal{A}_s} \sup_{\mu} |\mathscr{L}_{s,\alpha}[\widetilde{\Phi}_{\ell,\mu}](D) f| \|_{\ell^2(\Z^n)} \lesssim 2^{-\gamma s} \|f\|_{\ell^2(\Z^n)}. \]
To establish such decay for $p\in (1,\infty)$ we interpolate with the following estimate,
which is a direct consequence of Lemma \ref{lem:triv} (ii):
for every $\ell\in \mathcal{L}_2$,
\begin{equation}\label{eqn:intermedfreqpf1}
\|\sup_{\mu} \sup_{\alpha\in\mathcal{A}_s}  |\mathscr{L}_{s,\alpha}[\widetilde{\Phi}_{\ell,\mu}](D)f| \|_{\ell^p(\Z^n)}
\end{equation}
is dominated by 
\[ \|\sup_{\mu\in\R} |\mathcal{F}^{-1}_{\R^n}[\widetilde{\Phi}_{\ell,\mu}]| \|_{\ell^1(\Z^n)} \|f\|_{\ell^p(\Z^n)}. \]
If $J_{\ell,\mu}=\{j\}$, then
\[ \mathcal{F}^{-1}_{\R^n}[\widetilde{\Phi}_{\ell,\mu}](y) = e(\mu |y|^{2d}) K_j(y), \]
so that $\|\sup_{\mu\in\R} |\mathcal{F}^{-1}_{\R^n}[\widetilde{\Phi}_{\ell,\mu}]| \|_{\ell^1(\Z^n)}\approx 1$.
Therefore, 
\eqref{eqn:intermedfreqpf1} is bounded by a constant times $\|f\|_{\ell^p(\Z^n)}$, which suffices to establish the claim.
\end{proof}

\begin{remark}
There is an alternate argument for $p\ge 2$ that does not
make use of interpolation. 
Similarly as in \cite[\S 7.2]{KR},
one estimates
the term \eqref{eqn:intermedfreqpf1}
by appropriately subtracting the zero frequency, applying the Sobolev embedding trick, using Ionescu-Wainger theory and a version of Lemma \ref{prop:majorarc1}. One is then presented with a real-variable multiplier problem. 
A variant of a theorem due to Seeger \cite{See} allows us to reduce this to the estimate (in the case $n=d=1$)
\[ \| \Big(\int_1^2 \Big| \int_{\R} e(x\xi) e(\tau \xi^2) \widehat{f}(\xi) d\xi\Big|^2 d\tau\Big)^{1/2} \|_{L^p(\R)} \lesssim \|f\|_{L^p(\R)}, \]
which holds if $p\ge 2$ (see Lee, Rogers and Seeger \cite[Prop. 5.2]{LRS}).
We also refer to \cite{GRY17}, where similar arguments were used to obtain $L^p(V^r)$ inequalities associated with real-variable singular integrals of Stein-Wainger type.
This approach does not appear adequate to prove bounds for $p\in (1,2]$.
\end{remark}	

\begin{proof}[Proof of Lemma \ref{prop:Lemma7.2}]
As in \cite{KR}, the argument splits into two cases: either $J\le 2^{Cs}$ or $J>2^{Cs}$, with $C$ a large constant (if $s$ is large enough, then $C=2$ suffices).

{\em Case I: $J\le 2^{Cs}$}. The numerical inequality 
from Lemma \ref{lem:rmineq} yields a domination of the left-hand side by
\[ \sum_{l=0}^{Cs} \Big\| \Big( \sum_{\kappa\le 2^{2s-l}} \sup_{\alpha} |\mathscr{L}_{s,\alpha}[\widehat{K^{\kappa 2^l,(\kappa+1)2^l}}](D)f|^2  \Big)^{1/2}\Big\|_{\ell^p(\Z^n)} \]
\[ +\quad \| \sup_{\alpha} |\mathscr{L}_{s,\alpha}[\widehat{\mathcal{K}_1}](D)f|\|_{\ell^p(\Z^n)} \]
Lemma \ref{prop:Lemma7.1} and standard Calder\'on-Zygmund theory
allow us to bound this by $2^{-\gamma_p s} \|f\|_{\ell^p(\Z^n)}$.

{\em Case II: $J>2^{C s}$.} 
A key fact is that if $Q_s$ denotes the least common multiple of all integers in the range $[1, 2^s)$,
then, say, $2^J>Q_s^{100n}$ if $C>0$ is large enough.
By adding and subtracting $\widehat{\mathcal{K}}=\widehat{\mathcal{K}^{0,\infty}}$
and an application of Lemma \ref{prop:Lemma7.1} it suffices to control
\[ \|\sup_{\alpha\in\mathcal{A}_s} \sup_{J>2^{Cs}} |\mathscr{L}_{s,\alpha}[\widehat{\mathcal{K}^{J,\infty}}](D)f|\|_{\ell^p(\Z^n)}. \]
Next, let $\varphi$ denote a Schwartz function on $\R^n$ that is non-negative and satisfies $\int\varphi=1$ and
$\mathrm{supp}\;\widehat{\varphi}\subset \{ |\xi|<1/2\}$. Writing $\varphi_j(x)=2^{-jn}\varphi(2^{-j}x)$, it now suffices to estimate
\begin{equation}\label{eqn:7.2pf1}
\|\sup_{\alpha\in\mathcal{A}_s} \sup_{J>2^{Cs}} |\mathscr{L}_{s,\alpha}[\widehat{\varphi_J}\widehat{\mathcal{K}}](D)f|\|_{\ell^p(\Z^n)}.
\end{equation}
This is because the term
\[ \|\sup_{\alpha\in\mathcal{A}_s} \sup_{J>2^{Cs}} |\mathscr{L}_{s,\alpha}[\widehat{\mathcal{K}^{J,\infty}}-\widehat{\varphi_J}\widehat{\mathcal{K}}](D)f|\|_{\ell^p(\Z^n)} \]
can be dominated by
\begin{equation}\label{eqn:7.2pf2}
	\Big\|\Big( \sum_{J\ge 1} \sup_{\alpha\in\mathcal{A}_s} |\mathscr{L}_{s,\alpha}[\widehat{\mathcal{K}^{J,\infty}}-\widehat{\varphi_J}\widehat{\mathcal{K}}](D)f|^2 \Big)^{1/2}\Big\|_{\ell^p(\Z^n)}.
\end{equation}
Observe that
\[ \sum_{J\ge 1} \epsilon_J (\mathcal{K}^{J,\infty}-\varphi_J*\mathcal{K}) \]
is a Calder\'on-Zygmund kernel (its Fourier transform is bounded and it satisfies H\"ormander's condition), uniformly in all choices of $\epsilon_J\in\{\pm 1\}$.
Thus, by standard Calder\'on-Zygmund theory and Lemma \ref{prop:Lemma7.1}
we can bound \eqref{eqn:7.2pf2} by $2^{-\gamma_p} \|f\|_{\ell^p(\Z^n)}$ for all $p\in (1,\infty)$.
To estimate \eqref{eqn:7.2pf1} we use the fact that $\varphi_J$ is essentially unchanged when shifted by $u\in [Q_s]^n$.
Indeed, for every $u\in [Q_s]^n$,
\[	\|\sup_{\alpha\in\mathcal{A}_s} \sup_{J>2^{Cs}} |\mathscr{L}_{s,\alpha}[(\widehat{\varphi_J}-\widehat{\varphi_J(\cdot-u)})\widehat{\mathcal{K}}](D)f|\|_{\ell^p(\Z^n)} \]
is dominated by, say, $\lesssim 2^{-10s} \|f\|_{\ell^p(\Z^n)}$. This uses
$|u|2^{-J}\le Q_s^{-100n+1}\ll 2^{-100ns}$, the triangle inequality on the summation over $\beta$ and $\ell^p$ bounds for the Hardy--Littlewood maximal function and the Calder\'on-Zygmund kernel $\mathcal{K}$ (see \cite[\S 7.5]{KR} for more details).
Averaging over $u\in [Q_s]^n$ it now suffices to bound
\[ \Big(Q_s^{-n} \sum_{u\in [Q_s]^n} \|\sup_{\alpha\in\mathcal{A}_s} \sup_{J\ge 1} |\mathscr{L}_{s,\alpha}[\widehat{\varphi_J(\cdot-u)}\widehat{\mathcal{K}}](D)f|\|_{\ell^p(\Z^n)}^p \Big)^{1/p}. \]
The $p$th power of this expression equals
\[ Q_s^{-n} \sum_{u\in [Q_s]^n} \sum_{x\in \Z^n} \sup_{\alpha\in\mathcal{A}_s} \sup_{J\ge 1} |\sum_{\beta\in \frac1q [q]^n} S(\alpha,\beta) e(x\cdot \beta) (\varphi_J*\mathcal{K}_s*M_{-\beta} f)(x-u) |^p, \]
where $\mathcal{K}_s = \mathcal{F}_{\R^n}^{-1}[\widehat{\mathcal{K}}\cdot \chi_{s,M}]$ and $M_{-\beta} f(y)=e(-\beta y) f(y)$.
Changing variables $x\mapsto x+u$ and then $u\mapsto v-x$ and crucially using that $Q_s\beta\in \Z^n$ for all $\beta\in \frac1q \Z^n$ (with $q\in [2^{s-1}, 2^s)$),
we see that the previous is equal to
\[ Q_s^{-n} \sum_{v\in [Q_s]^n} \sum_{x\in \Z^n} \sup_{\alpha\in\mathcal{A}_s} \sup_{J\ge 1} |\sum_{\beta\in \frac1q [q]^n} S(\alpha,\beta) e(v\cdot \beta) (\varphi_J*\mathcal{K}_s*M_{-\beta} f)(x) |^p, \]
which equals
\[ Q_s^{-n} \sum_{v\in [Q_s]^n} \sum_{x\in \Z^n} \sup_{\alpha\in\mathcal{A}_s} \sup_{J\ge 1} |\varphi_J * \Big[\sum_{\beta\in \frac1q [q]^n} S(\alpha,\beta) e(v\cdot \beta) (\mathcal{K}_s*M_{-\beta} f) \Big] (x) |^p. \]
By $\ell^p$ boundedness of the discrete Hardy--Littlewood maximal function the previous is bounded by
\[ \lesssim Q_s^{-n}\sum_{v\in [Q_s]^n} \sum_{x\in \Z^n} \sup_{\alpha\in\mathcal{A}_s} |\sum_{\beta\in \frac1q [q]^n} S(\alpha,\beta) e(v\cdot \beta) (\mathcal{K}_s * M_{-\beta} f)(x) |^p. \]
Changing variables back, $v\mapsto u+x$ and then $x\mapsto x-u$ we arrive at
\[ Q_s^{-n}\sum_{u\in [Q_s]^n} \sum_{x\in \Z^n} \sup_{\alpha\in\mathcal{A}_s} |\sum_{\beta\in \frac1q [q]^n} S(\alpha,\beta) e(x\cdot \beta) (\mathcal{K}_s(\cdot-u) * M_{-\beta} f)(x) |^p. \]
For every fixed $u\in \Z^n$ we have by Lemma \ref{prop:Lemma7.1},
\[ \|\sup_{\alpha\in \mathcal{A}_s} |\mathscr{L}_{s,\alpha}[\mathcal{K}(\cdot-u)](D)f|\|_{\ell^p(\Z^n)} \lesssim 2^{-\gamma_p s}\|f\|_{\ell^p(\Z^n)}, \]
with constant not depending on $u$. This concludes the proof.
\end{proof}

\section{Reduction to major arc parameters: Proof of Proposition \ref{prop:tts1}}\label{sec:TTstar}
By interpolation with \eqref{eqn:trivunif}, we observe that it suffices to prove the claim for $p=2$.
To do this, we use the argument from our previous paper, see \cite[\S 4]{KR}. Here we will only describe the necessary changes to that argument. To begin, it suffices to determine $M>0$ so that for all functions $\lambda:\Z^n\to \R\setminus X_{j,M}$ we have
\[ \Big\| \sum_{y\in\Z^n} f(y) e(\lambda(x)|x-y|^{2d})K_j(x-y)\mathbf{1}_{B_j}(y) \Big\|_{\ell_x^2(\Z^n)} \lesssim j^{-\kappa} \|f\|_{\ell^2(\Z^n)}, \]
for all large enough $j$, where $B_j=\{y\in\Z^n\,:\,|y|\le 2^j\}$.
Denoting the operator in the last display by $T_{j,\lambda}$ we see that 
\[ T_{j,\lambda} (T_{j,\lambda})^* g(x) = \sum_{y\in\Z^n} g(y) \mathcal{K}^\sharp_{j,\lambda}(x,y), \]
where
\begin{equation}\label{eqn:ttsmainkernel}
\mathcal{K}^\sharp_{j,\lambda}(x,y) = \sum_{z\in\Z^n} e(\lambda(x)|z|^{2d}-\lambda(y)|y-x+z|^{2d})
\end{equation}
\[ \times K_j(z) \overline{K_j(y-x+z)}\mathbf{1}_{B_j}(x-z).\]
The support of $\mathcal{K}^\sharp_{j,\lambda}$ is contained in $B_{j+2}\times B_{j+2}$. Let
\[ E_{j,\lambda,\kappa} = \{ (x,y)\in\Z^n\times\Z^n\,:\,|\mathcal{K}^\sharp_{j,\lambda}(x,y)|\ge c_0 2^{-2jn} j^{-2\kappa} \}, \]
with $c_0$ to be determined.
We will show that there exists $C_1>0$ large enough so that for every $\kappa>1$ there exists $c_0>0$ so that for all $\lambda:\Z^n\to \R\setminus X_{j,{C_1 \kappa}}$ and $j\ge 1$,
\begin{equation}\label{eqn:ttsmainest}
|E_{j,\lambda,\kappa}| \lesssim 2^{nj} j^{-4\kappa}.
\end{equation}
With the pointwise estimate
\[ |\mathcal{K}^\sharp_{j,\lambda}(x,y)| \lesssim 2^{-2nj} j^{-2\kappa} \mathbf{1}_{B_{j+2}\times B_{j+2}}(x,y) + 2^{-2nj} \mathbf{1}_{E_{j,\lambda,\kappa}}(x,y), \]
the relation \eqref{eqn:ttsmainest} then implies
\[ \|\mathcal{K}^\sharp_{j,\lambda}\|_{\ell^2(\Z^n\times \Z^n)} \lesssim j^{-2\kappa}, \]
as long as $\lambda(y)\not\in X_{j,{C_1 \kappa}}$ for all $y\in\Z^n$. The Cauchy-Schwarz inequality and $\ell^2$ duality then give the desired estimate.
It remains to show \eqref{eqn:ttsmainest}.
The key to this is the following statement.
\begin{lem}
The constants $c_0$ and $\delta_0$ can be chosen such that for every $(x,y)\in E_{j,\lambda,\kappa}$, and $j$ large enough there exists a reduced rational $\frac{a}q$ with $q\le 2d\cdot j^{\delta_0}$ such that 
\begin{equation}\label{eqn:ttsmainpf2} |(x_1-y_1)\lambda(y)-\tfrac{a}q|\le 2^{-j(2d-1)} j^{\delta_0}.
\end{equation}
\end{lem}

\begin{proof}
The coefficient of $z_1^{2d-1}$ in the phase of \eqref{eqn:ttsmainkernel} is equal to $2d(x_1-y_1)\lambda(y)$. 
Applying Proposition \ref{prop:expsumlog} 
(with $\gamma=2\kappa$, $R=2^j$, $\alpha_0=(2d-1,0,\dots,0)$) we obtain constants $c_0>0$ and $\delta_0>0$ so that
if 
\begin{equation}\label{eqn:ttsmainpf1} |2d(x_1-y_1)\lambda(y)-\tfrac{a}{q}|\le \tfrac1{q^2},\quad q\le 2^{j(2d-1)} j^{-\delta_0}
\end{equation}
and $q\ge j^{\delta_0}$ hold for a reduced rational $\frac{a}q$, then
$|\mathcal{K}^\sharp_{j,\lambda}(x,y)|<c_0 2^{jn} j^{-2\kappa}.$
Dirichlet's approximation theorem on the other hand implies the existence of a reduced rational $\frac{a}q$ so that \eqref{eqn:ttsmainpf1} holds. 
Since $(x,y)\in E_{j,\lambda,\kappa}$ means that $|\mathcal{K}^\sharp_{j,\lambda}(x,y)|\ge c_0 2^{jn} j^{-2\kappa},$ we must have $q\le j^{\delta_0}$. Dividing through by $2d$ gives the claim.
\end{proof}

Observe that if we replace $\delta_0$ by any larger number, the conclusion continues to hold. In particular, we may assume that $\delta_0>4\kappa$.

Fix $(x',y)\in\Z^{n-1}\times \Z^n$ and let
\[ \mathcal{E} = \{x_1\in\Z\,:\,(x_1,x',y)\in E_{j,\lambda,\kappa}\}. \]
It will now suffice to determine $C_1>0$ so that for all $\kappa>1$ and $\lambda:\Z^n\to \R\setminus X_{j,{C_1 \kappa}}$ we have
\[  |\mathcal{E}| \le 2^j j^{-\delta_0}\le 2^{j} j^{-4\kappa}. \]
We argue by contradiction and assume
\[ |\mathcal{E}| > 2^{j} j^{-\delta_0}. \]
Then we follow exactly the same argument as in \cite[(4.9) etc.]{KR}, where every occurrence of $2^{\varepsilon_0 j}$ is to be replaced by $j^{\delta_0}$ and the role of \cite[(4.6)]{KR} is taken up by \eqref{eqn:ttsmainpf2} above.
This gives (assuming that $j$ is large enough) that
\[ \lambda(y) \in X_{j,{10 \delta_0}}, \]
resulting in a contradiction if $C_1 \kappa\ge 10\delta_0$, which holds as long as $C_1$ is large enough.

\section{Error estimate: Proof of Proposition \ref{prop:error}}\label{sec:errorest}

The key estimate to verify is that
\begin{equation}\label{eqn:errorestmain} |E_{j,\lambda,M}(\xi)|\lesssim 2^{-\delta j}
\end{equation}
for some constant $\delta>0$ (independent of $M$).
Also keeping in mind that $|\partial_\lambda E_{j,\lambda,M}(\xi)|\lesssim 2^{2dj}$ and using the fundamental theorem of calculus to estimate the supremum (precisely, using \cite[Lemma 5.1]{KR}), the inequality \eqref{eqn:errorestmain} implies
\[ \|\sup_{\lambda\in X_{j,M}} |E_{j,\lambda,M}(D)f|\|_{\ell^2(\Z^n)} \lesssim j^{\frac32 M} 2^{-\frac{\delta}2 j} \|f\|_{\ell^2(\Z^n)}. \] 
Interpolating this with the trivial estimate 
\[ \|\sup_{\lambda\in X_{j,M}} |E_{j,\lambda,M}(D)f|\|_{\ell^p(\Z^n)} \lesssim  \|f\|_{\ell^p(\Z^n)}, \]
valid for all $p\in [1,\infty]$, we obtain the claim. Finally, \eqref{eqn:errorestmain} follows from the same argument as in \cite[\S 5]{KR}; the only change being that every occurrence of $2^{j\varepsilon_1}$ is replaced by $j^M$ (the term $2^{j\varepsilon_2}$ remains unchanged; the case distinction $1\le s_0\le \varepsilon_1 j$ vs. $\varepsilon_1 j\le s_0\le \varepsilon_2 j$ becomes $2^{s_0}\le j^M$ vs. $j^M\le 2^{s_0}\le 2^{\varepsilon_2 j}$).

\newcommand{\etalchar}[1]{$^{#1}$}

\end{document}